\documentclass[12pt]{amsart}
\usepackage{amsmath}
\usepackage{amsfonts}
\usepackage{amssymb}

\newcommand\eps{\varepsilon}

\newcommand\Z{{\mathbb{Z}}}

\newcommand\supp{{\operatorname{Supp}}}

\theoremstyle{plain}
  \newtheorem{theorem}[subsection]{Theorem}
  
  \newtheorem{proposition}[subsection]{Proposition}
  \newtheorem{lemma}[subsection]{Lemma}
  \newtheorem{corollary}[subsection]{Corollary}

\theoremstyle{remark}
  \newtheorem{remark}[subsection]{Remark}

\theoremstyle{definition}
  \newtheorem{definition}[subsection]{Definition}

\begin{document}

\title[]{Analyticity of extremisers to the Airy Strichartz inequality}
\author{Dirk Hundertmark}
\address{Department of Mathematics, University of Illinois at Urbana-Champaign, Urbana, IL 61801}
\email{dirk@math.uiuc.edu}
\author{Shuanglin Shao}
\address{Institute for Mathematics and its applications, University of Minnesota, Minneapolis, MN 55455}
\email{slshao@ima.umn.edu}

\vspace{-0.1in}
\date{\today}
\begin{abstract}We prove that there exists an extremal function to the Airy Strichartz inequality
$$
\|e^{-t\partial_x^3} f\|_{L^8_{t,x}(\mathbb{R}\times \mathbb{R})} \le C\|f\|_{L^2(\mathbb{R})},
$$
by using the linear profile decomposition. Furthermore we show that, if $f$ is an extremiser, then $f$ is extremely fast decaying in Fourier space and so $f$ can be extended to be an entire function on the whole complex domain. The rapid decay of the Fourier transform of extremisers is established with a bootstrap argument which relies on a refined bilinear Airy Strichartz estimate and a weighted Strichartz inequality.
\end{abstract}

\maketitle

\section{Introduction}
It is well known that the (generalized) Korteweg-de Vries equations (KdV or gKdV) are good approximations to the evolution of waves on shall water surface \cite{Craig:1985:water-wave-KdV, Schneider-Wayne:2000:KdV-water-wave-zero-surface-tension, Schneider-Wayne:2002:KdV-water-wave-surface-tension}:
\begin{equation}\label{eq:gkdv}
\partial_t u+\partial_x^3 u\pm \partial_x(u^p)=0
\end{equation}
for $p\ge 2$. The linear form is the Airy equation
\begin{equation}\label{eq:airy}
\partial_t u+\partial_x^3 u=0.
\end{equation} In general, for an initial data $u(0)=f(x)$ the solution $e^{-t\partial_x^3} f$ to the Airy solution can be expressed as
\begin{equation}\label{eq-3}
e^{-t\partial_x^3} f(x):=(2\pi)^{-1/2}\int_\mathbb{R} e^{ixk+itk^3}\widehat{f}(k)dk.
\end{equation}
The linear Strichartz inequality for \eqref{eq:airy} asserts that
\begin{equation}\label{eq:strichartz}
 \|D^{\alpha} e^{-t\partial_x^3}f\|_{L^q_tL^r_x} \lesssim \|f\|_2,
\end{equation}
for $-\alpha+\frac 3q+\frac 1r=\frac 12$ and $-1/2<\alpha \le 1/q$,
see \cite[Theorem 2.1]{Kenig-Ponce-Vega:1991:dispersive-estimates}.
When $\alpha=1/q$, the inequality above is called ``endpoints" while ``nonendpoints" for $\alpha<1/q$. It plays an important role in establishing local or global wellposedness theory for the Cauchy problem of \eqref{eq:gkdv}, see for instance \cite{Kenig-Ponce-Vega:1991:dispersive-estimates, Tao:2006-CBMS-book}.  In this paper, we study the the following symmetrical Strichartz inequality
\begin{equation}\label{eq:airy-strichartz}
\|e^{-t\partial_x^3} f\|_{L^8_{t,x}(\mathbb{R}\times \mathbb{R})} \le C\|f\|_{L^2(\mathbb{R})},
\end{equation}
and consider ``extremisers"  for \eqref{eq:airy-strichartz}: the existence of extremisers and characterization of some of their properties.

To begin with, we denote the optimal constant for \eqref{eq:airy-strichartz} by $\mathcal{A}$:
\begin{equation}\label{eq:best-constant}
 \mathcal{A}:=\sup\{\|e^{-t\partial_x^3}f\|_{L^8_{t,x}}:\,\|f\|_2=1\}.
\end{equation}
A simple argument, together with \eqref{eq:strichartz} shows that $\mathcal{A}<\infty$, see the proof of Theorem \ref{thm:profile L8L8}.
\begin{definition}
A function $f\in L^2$ is said to be an extremiser for \eqref{eq:airy-strichartz} if $f$ is not equal to the zero function a.e.\ and
\begin{equation}\label{eq-4}
\|e^{-t\partial_x^3} f\|_{L^8_{t,x}(\mathbb{R}\times \mathbb{R})} = \mathcal{A} \|f\|_{L^2(\mathbb{R})}.
\end{equation}
\end{definition}
The first result is the following theorem.
\begin{theorem}\label{thm-existence} There exists an extremal function $f\in L^2$ for the Airy Strichartz inequality \eqref{eq:airy-strichartz}.
\end{theorem}
This theorem is proven in Section \ref{sec:existence}. The proof makes use of the linear profile decomposition for the Airy evolution operator $e^{-t\partial_x^3}$ acting on a bounded sequence of $\{f_n\}\in L^2$, which we develop in Section \ref{sec:profile} based on the previous result in \cite{Shao:2008:linear-profile-Airy-Maximizer-Airy-Strichartz}. In \cite{Shao:2008:maximizers-Strichartz-Sobolev-Strichartz}, the profile decomposition for the Schr\"odinger equation developed in \cite{Begout-Vargas:2007:profile-schrod-higher-d} was used to prove the existence of extremisers to the Strichartz inequality for the Schr\"odinger equation in higher dimensions. The profile decomposition can be viewed as a manifestation of the idea of ``concentration-compactness",  see P.-L. Lions \cite{Lions-1984-cc-locally-compact-I, Lions-1984-cc-locally-compact-II, Lions-1985-cc-limit-case-I, Lions-1985-cc-limit-case-II}.

\begin{remark} Theorem \ref{thm-existence} is different from that in \cite{Shao:2008:linear-profile-Airy-Maximizer-Airy-Strichartz} where a dichotomy result is obtained on the existence of extremisers to the Strichartz inequality $\|e^{-t\partial_x^3}D^{1/6}f\|_{L^6_{t,x}}\le C \|f\|_{L^2}$, which is the symmetric ``endpoint" Strichartz inequality; in other words, for this Strichartz inequality, either an extremiser exists or a sequence of modulated Gaussians approximates to the extremiser. The dichotomy is due to the presence of highly oscillatory terms in the refined profile decomposition, see Theorem \ref{thm:Airy-prof}. Another instance of a dichotomy result on extremisers to a Strichartz-type inequality is in \cite{Jiang-Pausader-Shao:2010:4th-NLS}. The presence of highly oscillatory terms in the profile decomposition is not a problem for the existence of extremisers if the equation is invariant under boosts, i.e., shifts in momentum (or Fourier) space, which is the case for the Schr\"odinger and wave equations. The Airy equation \eqref{eq:airy} is, however, \emph{not invariant} under shifts in momentum space. Hence to get the existence of maximizers for \eqref{eq:airy-strichartz} we need a profile decomposition which avoids highly oscillatory terms, which is done in Theorem \ref{thm:profile L8L8}.
\end{remark}

Extremisers to the Strichartz inequality for the Schr\"odinger equation and the wave equation have been studied intensively recently. For the Strichartz inequality for the Schr\"odinger equation, Kunze \cite{Kunze:2003:maxi-strichartz-1d} proved the existence of extremisers to the one dimensional Strichartz inequality by establishing that any nonnegative extremizing sequence converges strongly an extremiser in $L^2$ up to the natural symmetries of the inequality. In the lower dimensional case, the existence of extremisers was shown by Foschi \cite{Foschi:2007:maxi-strichartz-2d} and Hundertmark, Zharnitsky \cite{Hundertmark-Zharnitsky:2006:maximizers-Strichartz-low-dimensions}: Gaussians are extremisers, which are unique up to the natural symmetries of the inequality. Later works devoted to the study of the Strichartz inequality for the Schr\"odinger equation with different emphases include \cite{Bennett-Bez-Carbery-Hundertmark:2008:heat-flow-of-strichartz-norm, Carneiro-2009-sharp-strichartz, Christ-Quilodran:Gaussians-rarely-extremize-non-L2-restriction-paraboloids}. To the best of our knowledge, we remark that all the previous known methods do not seem to be adapted directly to finding the explicit form of ``extremisers" to \eqref{eq:airy-strichartz} in our setting. For extremisers to the Strichartz inequality for the wave equation, see \cite{Foschi:2007:maxi-strichartz-2d, Bulut:2009:maximizer-wave}.

Closely related to the Strichartz inequality for the Schr\"odinger equations, Christ and Shao  \cite{Christ-Shao:extremal-for-sphere-restriction-I-existence, Christ-Shao:extremal-for-sphere-restriction-II-characterizations} studied ``extremisers" to an adjoint Fourier restriction inequality for the sphere, namely the Tomas-Stein inequality $L^2(S^2)\to L^4_{x}(\mathbb{R}^3)$ for two dimensional sphere $S^2$. Although the Strichartz inequality for the Schr\"odinger equation can be viewed as an adjoint Fourier restriction inequality for the paraboloid, the situation for the sphere is different from the paraboloid case due to the nonlocal property and the lack of scaling symmetry of the adjoint Fourier restriction operator: $L^2(S^2)\to L^4_{x}(\mathbb{R}^3)$. However, among other things, they were able to show that there exists an extremal by proving that any extremising sequence of nonnegative functions in $L^2(S^2)$ has a strongly convergent subsequence.
 For existence of quasiextremals and extremisers to the convolution inequality with the surface measure on the paraboloid or the sphere, see \cite{Christ:quasiextremiser-radon-like-transform, Christ:extremiser-radon-like-transform, Stovall:quasiextremals-to-convolution-sphere}.

Next we turn to the characterization of the extremisers to \eqref{eq:airy-strichartz} from studying the corresponding generalized Euler-Lagrange equation:
\begin{equation}\label{eq:euler-lagrange}
\omega f= \int e^{t\partial_x^3}\bigl[|e^{-t\partial_x^3}f|^6e^{-t\partial_x^3}f\bigr]dt,
\end{equation}
where $\omega$ is a Lagrange multiplier, which for extremisers $f$ is given by $\omega= \mathcal{A}^8 \|f\|^6_2$ where $\mathcal{A}$ is the optimal constant defined in \eqref{eq:best-constant}. The Euler-Lagrange equation \eqref{eq:euler-lagrange} can be established by a standard variational argument. Traditionally, once the existence of an extemiser has been shown its properties are deduced from studying the associated Euler-Lagrange equation. Note that in our case \eqref{eq:euler-lagrange} is a highly non-linear and non-local equation, which makes this a rather non-trivial task. Nevertheless the following strong regularity result for extremisers holds.
\begin{theorem}\label{thm-exp-decay}
For any extremiser $f$ to the Airy Strichartz inequality \eqref{eq:airy-strichartz} there exists $\mu_0>0$ such that
 \begin{equation}\label{eq-5}
  k\mapsto e^{\mu_0|k|^3} \widehat{f}(k)\in L^2;
 \end{equation}
where $\hat{f}$ is the Fourier transform of $f$. In particular, $f$ can be extended to be an entire function on the complex plane.
\end{theorem}
The proof of this theorem is based on a bootstrap argument, which relies on a refined bilinear Strichartz inequality for Airy operator $e^{-t\partial_x^3}f$, and a weighted Strichartz inequality. The argument uses some ideas similar to Erdogan, Hundertmark and Lee \cite{Erdogan-Hundertmark-Lee:2010:exponential-decay-dispersion-management-solitons}, which in turn is based in part on \cite{Hundertmark-Lee:2009:decay-smoothness-for-solutions-dispersion-managed-NLS}.
In \cite{Erdogan-Hundertmark-Lee:2010:exponential-decay-dispersion-management-solitons}, it is shown that solutions to the dispersion managed non-linear Schr\"odinger equation in the case of zero residual dispersion are exponentially fast decaying not only in the Fourier space but also in the spatial space. The fact that \cite{Erdogan-Hundertmark-Lee:2010:exponential-decay-dispersion-management-solitons} also establishes decay in the spatial space is essentially due to the fact that the linear Schr\"odinger operator $e^{it\Delta}$ involved enjoys an identity
\begin{equation}\label{eq-11}
e^{it\Delta}f(x)=(2\pi)^{-d/2} \int_{\mathbb{R}^d} e^{ix\xi+it|\xi|^2}\hat{f}(\xi)d\xi=Ct^{-d/2} \int_{\mathbb{R}^d} e^{\frac {|x-y|^2}{4it}}f(y)dy, \text{ for some }C>0,
\end{equation}
which enables one to obtain the decay in the spatial space from that on the Fourier side. There is no such identity for the Airy operator and thus our Theorem \ref{thm-exp-decay} gives decay only in Fourier space. On the other hand, the decay given by Theorem \ref{thm-exp-decay} is much more rapid than even Gaussian decay.

The organization of the paper is as follows. In Section \ref{sec:profile}, we establish the linear profile decomposition. In Section \ref{sec:existence}, we show the existence of extremisers to the Airy Strichartz inequality $L^2\to L^8_{t,x}$. In Section \ref{sec:analyticity}, we show that any solution to the generalized Euler-Lagrange equation, which includes the extremiser as a special case, obeys a bound of the form \eqref{eq-5} and can be extended to be analytic on the complex plane. It is proven by assuming an important bootstrap lemma, which we establish in Section \ref{sec:bootstrap}.

\section{The linear profile decomposition}\label{sec:profile}
Recall from the introduction, we will use the linear profile decomposition for the Airy evolution operator $e^{-t\partial_x^3}$ for $L^2$ initial data to prove the existence of extremisers for \eqref{eq:airy-strichartz}.  Roughly speaking, the linear profile decomposition is to investigate the general structure of solutions $\{e^{-t\partial_x^3}f_n\}$ for bounded $\{f_n\}\in L^2$, and aims to compensate for the loss of compactness of solution operator caused by the symmetries of the equation, \cite{Lions-1984-cc-locally-compact-I}. For a sequence $\{e^{-t\partial_x^3}f_n\}$, it is expected to be written as a superposition of concentrating waves, ``profiles" plus an negligible reminder term; the interaction of the profiles is small, see the precise statements in Theorem \ref{thm:Airy-prof} and Theorem \ref{thm:profile L8L8}. The profile decomposition for the nonlinear wave and Schr\"odinger equation, and the gKdV equation have been developed in \cite{Bahouri-Gerard:1999:profile-wave, Begout-Vargas:2007:profile-schrod-higher-d, Carles-Keraani:2007:profile-schrod-1d, Keraani:2001:profile-schrod-H^1, Merle-Vega:1998:profile-schrod, Shao:2008:linear-profile-Airy-Maximizer-Airy-Strichartz}. To prepare for the linear profile decomposition theorem for the Airy evolution operator in the Strichartz norm $\|u\|_{L^8_{t,x}}$ needed in this paper, we recall two definitions from \cite{Shao:2008:linear-profile-Airy-Maximizer-Airy-Strichartz}.
\begin{definition}For any phase $\theta\in \mathbb{R}/2\pi\Z$, position $x_0\in \mathbb{R}$ and scaling parameter $h_0>0$, we define the unitary transform $g_{\theta,x_0,h_0}:L^2\to L^2$ by the formula
$$[g_{\theta,x_0,h_0}f](x):=\frac {1}{h_0^{1/2}}e^{i\theta}f(\frac {x-x_0}{h_0}).$$
\end{definition}We let $G$ be the collection of such transformations. It is easy to see that $G$ is a group which preserves the $L^2$ norm.

\begin{definition}\label{def-ortho}For $j\neq k$, two sequences $\Gamma_n^j:=(h_n^j,\xi_n^j, x_n^j,t_n^j)_{n\ge 1}$ and $\Gamma_n^k:=(h_n^k,\xi_n^k, x_n^k,t_n^k)_{n\ge 1}$ in $(0,\infty)\times \mathbb{R}^3$ are orthogonal if there holds,
\begin{align}
&\text{either } \limsup_{n\to \infty}\left(\dfrac {h_n^j}{h_n^k}+\dfrac {h_n^k}{h_n^j}+h_n^j|\xi_n^j-\xi_n^k|\right)=\infty,\\
&\text{or }(h_n^j,\xi_n^j)=(h_n^k, \xi_n^k)\text{ and }\\
&\quad \limsup_{n\to \infty}\left(\dfrac{|t_n^k-t_n^j|}{(h_n^j)^3}+\dfrac {3|(t_n^k-
t_n^j)\xi_n^j|}{(h_n^j)^2}+\dfrac{|x_n^j-x_n^k+3(t_n^j-t_n^k)
(\xi_n^j)^2|}{h_n^j}\right)=\infty \nonumber.
\end{align}
\end{definition}
Let $D^{\alpha}$, $\alpha\in \mathbb{R}$, be the fractional derivative operator defined in terms of the Fourier multiplier, $\widehat{D^\alpha f}=|\xi|^{\alpha}\hat{f}$. We state the following linear profile decomposition in the Strichartz norm $\|D^{1/6}\cdot\|_{L^6_{t,x}}$ from \cite{Shao:2008:linear-profile-Airy-Maximizer-Airy-Strichartz}.
\begin{theorem}\label{thm:Airy-prof}
Let $(f_n)_{n\ge 1}$, $f_n:\mathbb{R}\to \mathbb{C}$, be a sequence of functions satisfying $\|f_n\|_{L^2_{t,x}}\le 1$.  Then up to a subsequence, there exists a sequence of $L^2$ functions $(\phi^j)_{j\ge 1}: \mathbb{R}\to \mathbb{C}$ and a family of pairwise orthogonal sequences
$\Gamma_n^j=(h_n^j,\xi_n^j,x_n^j,t_n^j)\in (0,\infty)\times \mathbb{R}^3$ such that, for any $l\ge 1$,
there exists an $L^2$ function $w_n^l: \mathbb{R}\to \mathbb{C}$ satisfying
\begin{equation}\label{eq:prof}
f_n=\sum_{1\le j\le l, \xi_n^j\equiv 0 \atop \text{ or } |h_n^j\xi_n^j|\to \infty} e^{t_n^j\partial_x^3}g_n^j[e^{i(\cdot)h_n^j\xi_n^j}\phi^{j}]+w_n^l,
\end{equation}
where $g_n^j:=g_{0,x_n^j,h_n^j}\in G$ and \begin{equation}\label{eq:err}
 \limsup_{l\to \infty}\limsup_{n\to \infty} \|D^{1/6}e^{-t\partial_x^3}w_n^l\|_{L^6_{t,x}}=0.
\end{equation}
Moreover, for every $l\ge 1$, \begin{equation}\label{eq:almost-ortho}
 \limsup_{n\to \infty} \left|\|f_n\|^2_2-\left(\sum_{j=1}^l\|\phi^j\|^2_2
 +\|w_n^l\|^2_2\right)\right|=0.
\end{equation}
\end{theorem}

As a consequence of this theorem, we can develop a linear profile decomposition in the Airy-Strichartz norm $\|\cdot\|_{L^8_{t,x}}$, where the highly oscillatory terms $e^{ixh_n^j\xi_n^j}\phi^j(x)$ with $|h_n^j\xi_n^j|\to\infty$ disappear.
\begin{theorem}\label{thm:profile L8L8}
Let $(f_n)_{n\ge 1}$, $f_n:\mathbb{R}\to \mathbb{C}$, be a sequence of functions satisfying $\|f_n\|_2\le 1$.  Then up to a subsequence, there exists a sequence of $L^2$ functions $(\phi^j)_{j\ge 1}: \mathbb{R}\to \mathbb{C}$ and a family of parameters $\Gamma_n^j=(h_n^j,x_n^j,t_n^j)\in (0,\infty)\times \mathbb{R}^2$ such that, for any $l\ge 1$,
there exists an $L^2$ function $w_n^l: \mathbb{R}\to \mathbb{C}$ satisfying
\begin{equation*}
f_n=\sum_{1\le j\le l} e^{t_n^j\partial_x^3}g_n^j(\phi^j)+w_n^l,
\end{equation*}
where $g_n^j:=g_{0,x_n^j,h_n^j}\in G$ and
\begin{equation}\label{eq:error}
 \limsup_{l\to \infty}\limsup_{n\to \infty} \|e^{-t\partial_x^3}w_n^l\|_{L^8_{t,x}}=0,
\end{equation}
and for $j\neq k$,
\begin{equation}\label{eq:ortho-parameter}
\limsup_{n\to\infty} \left( \frac {h_n^j}{h_n^k}+\frac {h_n^k}{h_n^j} +\frac {|t_n^j-t_n^k|}{(h_n^j)^3}+\frac {|x_n^j-x_n^k|}{h_n^j}\right)=\infty.
\end{equation}
Moreover, we have two orthogonality results: for every $l\ge 1$,
\begin{align}
 &\limsup_{n\to \infty} \left|\|f_n\|^2_2-\left(\sum_{j=1}^l\|\phi^j\|^2_2
 +\|w_n^l\|^2_2\right)\right|=0. \label{eq:orthogonal1}\\
 &\limsup_{n\to\infty}\left| \|\sum_{1\le j\le l} e^{-(t-t_n^j)\partial_x^3}g_n^j(\phi^j)\|^8_{L^8_{t,x}}-
 \sum_{1\le j\le l} \|e^{-t\partial_x^3}\phi^j\|^8_{L^8_{t,x}}\right|=0. \label{eq:orthogonal2}
\end{align}
\end{theorem}

\begin{remark}\label{re-4} By \eqref{eq:orthogonal1} we have
 \begin{equation*}
  \sum_{j=1}^l \|\phi^j\|_2^2
  \le
  \liminf_{n\to\infty}\big(\sum_{j=1}^l \|\phi^j\|_2^2 +\|w_n^j\|_2^2 \big)
  \le
  \liminf_{n\to\infty} \|f_n\|_2^2 \le 1
 \end{equation*}
for any $l\in\mathbb{N}$. Hence $\sum_{j=1}^\infty \|\phi^j\|_2^2\le 1$.
\end{remark}
\begin{proof} This argument consists of three steps. We first see that the error term $w_n^l$ still converges to zero in this new Strichartz norm $\|\cdot\|_{L^8_{t,x}}$. Indeed, by the Sobolev embedding,
$$\| e^{-t\partial_x^3} u_0\|_{L^8_{t,x}}\le C\|D^{1/6} e^{-t\partial_x^3} u_0\|_{L^6_{t,x}};$$
so an application of \eqref{eq:err} yields that
$$\limsup_{l\to\infty}\limsup_{n\to\infty} \| e^{-t\partial_x^3} w_n^l\|_{L^8_{t,x}}=0.$$
Secondly we claim that, for $1\le j\le l$, when $\lim_{n\to\infty}h_n^j\xi_n^j=\infty$,
 \begin{equation}\label{eq-10}
 \lim_{n\to\infty}\|e^{-(t-t_n^j)\partial_x^3}g_n^j[e^{i(\cdot)h_n^j\xi_n^j}\phi^{j}]\|_{L^8_{t,x}}=0.
 \end{equation}
It shows that the highly oscillatory terms can be reorganized into the error term. To show \eqref{eq-10}, by using the symmetries, we reduce to prove
  \begin{equation}\label{eq-12}
  \lim_{N\to\infty}\|e^{-t\partial_x^3}[e^{i(\cdot)N}\phi]\|_{L^8_{t,x}}=0.
  \end{equation}
We may assume $\phi\in \mathcal{S}$, the set of Schwartz functions, and that $\phi$ has the compact Fourier support $(-1,1)$.
$$e^{-t\partial_x^3}[e^{i(\cdot)N}\phi](x)=e^{ixN+itN^3}\int e^{i(x+3tN^2)\xi+i3Nt\xi^2+it\xi^3}\widehat{\phi}(\xi)d\xi.$$
Setting $x':=x+3tN^2$ and $t':=3Nt$, we have,
$$\lim_{N\to\infty}\|e^{-t\partial_x^3}[e^{i(\cdot)N}\phi]\|_{L^8_{t,x}}=cN^{-1/8}
\|\int e^{ix'\xi+it'\xi^2+i\frac {t'}{3N}\xi^3}\widehat{\phi} d\xi\|_{L^8_{t',x'}}$$ for some $c>0$. Then the dominated convergence theorem yields
$$\lim_{N\to\infty}\|\int e^{ix'\xi+it'\xi^2+i\frac {t'}{3N}\xi^3}\widehat{\phi} d\xi\|_{L^8_{t',x'}}=
\|e^{-it\partial_x^2} \phi^j\|_{L^8_{t,x}}.$$
 Here $e^{-it\partial_x^2}$ denotes the Schr\"odinger evolution operator defined via \begin{equation*}
e^{-it\partial_x^2}f(x):=\int e^{ix\xi+it|\xi|^2}\widehat{f}(\xi)d\xi.
\end{equation*} Indeed, $$\int e^{ix'\xi+it'\xi^2+i\frac {\xi^3}{3N}}\widehat{\phi}(\xi)d\xi\to e^{-it'\partial_x^2}\phi^j(x'),\,a.e., $$ and by using \cite[Corollary, p.334]{Stein:1993} or integration by parts, $$ \left|\int e^{ix'\xi+it'\xi^2+i\frac {t'\xi^3}{3N}}\widehat{\phi}(\xi)d\xi\right|\le C_{\phi^j} B(t',x')$$ for $n$ large enough but still uniform in $n$. Here
\begin{equation*}
 B(t',x')=\begin{cases} (1+|t'|)^{-1/2}\le C\bigl[(1+|x'|)(1+|t'|)\bigr]^{-1/4}, &\text{for } |x'|\le 6|t'|,\\
 (1+|x'|)^{-1}\le C\bigl[(1+|x'|)(1+|t'|)\bigr]^{-1/2}, &\text{for } |x'|>6|t'|. \end{cases}
 \end{equation*}
It is easy to observe that $B\in L^8_{t',x'}$. Then \eqref{eq-12} follows immediately.

Finally we claim that, for $j\neq k$,
$$\lim_{n\to\infty}\|e^{-(t-t_n^j)\partial_x^3}g_n^j(\phi^j)
e^{-(t-t_n^k)\partial_x^3}g_n^k(\phi^k)\|_{L^4_{t,x}}=0.$$
This is a consequence of the orthogonality condition \eqref{eq:ortho-parameter}, whose proof is a special case of Lemma  \ref{le:orthogonality-in-strichartz} below. The remaining conclusions in Theorem \ref{thm:profile L8L8} follow from Theorem \ref{thm:Airy-prof} accordingly.
\end{proof}

\begin{remark}\label{re-3}
A linear profile decomposition for all non-endpoint Airy Strichartz inequalities can be established by using the first two observations in the previous lemma and Lemma \ref{le:orthogonality-in-strichartz}. The statement is similar to Theorem \ref{thm:profile L8L8} and so we omit the details.
\end{remark}

\begin{lemma}\label{le:orthogonality-in-strichartz}
When $-\alpha+\frac 3q+\frac 1r=\frac 12, -1/2<\alpha<\frac 12$. Then for $j\neq k$,
\begin{equation}\label{eq:loc-1}
\lim_{n\to\infty}\|e^{-(t-t_n^j)\partial_x^3} D^\alpha g_n^j (\phi^j) e^{-(t-t_n^k)\partial_x^3}D^\alpha g_n^k (\phi^k)\|_{L^{q/2}_t L^{r/2}_x}=0
\end{equation}provided that $\{(h_n^j,x_n^j,t_n^j)\}$ and $\{(h_n^k,x_n^k,t_n^k)\}$ satisfies the orthogonality condition in \eqref{eq:ortho-parameter}.
\end{lemma}
\begin{proof} We will prove \eqref{eq:loc-1} by studying \eqref{eq:ortho-parameter} case by case.

\text{\bf{Case I.}} Assume $\limsup_{n\to\infty}\frac {h_n^j}{h_n^k}+\frac {h_n^k}{h_n^j} =\infty$.
For any $R>0$, we define
\begin{align*}
\Omega_n^j(R):&=\{(t,x): \frac {|x-x_n^j|}{h_n^j}+\frac {|t-t_n^j|}{(h_n^j)^3}\le R\},\\
\Omega_n^k(R):&=\{(t,x): \frac {|x-x_n^k|}{h_n^k}+\frac {|t-t_n^k|}{(h_n^k)^3}\le R\},\\
\bigl(\Omega_n^j\bigr)^c:&=\mathbb{R}^2\setminus \Omega_n^j(R),\, \bigl(\Omega_n^k\bigr)^c:=\mathbb{R}^2\setminus \Omega_n^k(R).
\end{align*}
By using H\"older's inequality and the Strichartz inequality followed by a change of variables, we have
\begin{align*}
&\|e^{-(t-t_n^j)\partial_x^3} D^\alpha g_n^j (\phi^j) e^{-(t-t_n^k)\partial_x^3}D^\alpha g_n^k (\phi^k)\|_{L^{q/2}_t L^{r/2}_x\bigl((\Omega_n^j)^c\bigr)}\\
&\le C\|e^{-(t-t_n^j)\partial_x^3} D^\alpha g_n^j (\phi^j)\|_{L^q_tL^r_x\bigl((\Omega_n^j)^c\bigr)}\|e^{-(t-t_n^k)\partial_x^3}D^\alpha g_n^k (\phi^k)\|_{L^q_tL^r_x}\\
&\le C (h_n^j)^{-1/2-\alpha} \|e^{-\frac {t-t_n^j}{(h_n^j)^3}\partial_x^3}(D^\alpha \phi^j)(\frac {x-x_n^j}{h_n^j})\|_{L^q_tL^r_x\bigl((\Omega_n^j)^c\bigr)}\|\phi^k\|_2\\
&\le C\|\phi^k\|_2 \|e^{-t\partial_x^3} D^\alpha(\phi^j)\|_{L^q_tL^r_x(\{|x|+|t|\ge R\})}.
\end{align*} The latter integral converges to zero when $R$ goes to infinity from the dominated convergence theorem. So we can choose a  sufficiently large $R>0$ such that
 $$\|e^{-(t-t_n^j)\partial_x^3} D^\alpha g_n^j (\phi^j) e^{-(t-t_n^k)\partial_x^3}D^\alpha g_n^k (\phi^k)\|_{L^{q/2}_t L^{r/2}_x\bigl((\Omega_n^j)^c\bigr)} $$
 as small as we want. Likewise for $\|e^{-(t-t_n^j)\partial_x^3} D^\alpha g_n^j (\phi^j) e^{-(t-t_n^k)\partial_x^3}D^\alpha g_n^k (\phi^k)\|_{L^{q/2}_t L^{r/2}_x\bigl((\Omega_n^k)^c\bigr)}$. So fixing a large $R$, we may restrict our attention onto $\Omega_n^j\cap \Omega_n^k$. We aim to show that the integral on $\Omega_n^j\cap \Omega_n^k$ converges to zero when $n$ goes to infinity. Indeed, by using trivial $L^\infty_{t,x}$ bounds on $e^{-(t-t_n^j)\partial_x^3} D^\alpha g_n^j (\phi^j)$ and $e^{-(t-t_n^k)\partial_x^3}D^\alpha g_n^k (\phi^k)$, we see that
\begin{align*}
\|e^{-(t-t_n^j)\partial_x^3} D^\alpha g_n^j (\phi^j) e^{-(t-t_n^k)\partial_x^3}D^\alpha g_n^k (\phi^k)\|_{L^{q/2}_t L^{r/2}_x(\Omega_n^j\cap \Omega_n^k})\\
\le C(h_n^jh_n^k)^{-1/2-\alpha}\min\{(h_n^j)^{6/q+2/r},
(h_n^k)^{6/q+2/r}\}\\
\le C\min\{(\frac {h_n^j}{h_n^k})^{1/2+\alpha},
(\frac {h_n^k}{h_n^j})^{1/2+\alpha}\}\to 0
\end{align*} as $n$ goes to infinity. Note that $C>0$ depending on $R,\|\widehat{\phi^j}\|_{L^1}, $ and $\|\widehat{\phi^k}\|_{L^1}$. Thus \eqref{eq:loc-1} is obtained, which completes the proof of \text{\bf Case I.}

\text{\bf Case II.} Now we may assume that $h_n^j=h_n^k$ for all $n$, we are left with the case where $$\limsup_{n\to\infty}\frac {|x_n^j-x_n^k|}{h_n^j}+\frac {|t_n^j-t_n^k|}{(h_n^j)^3}=\infty.$$
We change variables $x'=\frac {x-x_n^k}{h_n^k}$ and $t'=\frac {t-t_n^k}{(h_n^k)^3}$ and see that we need to show that
$$\|e^{-(t'+\frac {t_n^k-t_n^j}{(h_n^j)^3})}(D^\alpha \phi^j)(x'+\frac {x_n^k-x_n^j}{h_n^j})e^{-t'\partial_x^3}(D^\alpha\phi^k)(x') \|_{L^{q/2}_{t'}L^{r/2}_{x'}} \to 0 $$
as $n\to\infty$. We define
\begin{align*}
\Omega^k(R):&=\{(t,x): |t'|+|x'|\le R\},\\
\Omega_n^j(R):&=\{(t,x): \left|x'+\frac {x_n^k-x_n^j}{h_n^j}\right|+\left|t'+\frac {t_n^j-t_n^k}{(h_n^j)^3}\right|\le R\}.
\end{align*} As proving \text{\bf Case I}, we may reduce to the domain $\Omega^k\cap \Omega_n^j$. While for this case, we observe that, for any fixed large $R>0$, $$|\Omega^k\cap \Omega_n^j|\to 0, \text{ as } n\to\infty.$$ This, together with the $L^\infty_{t,x}$ bounds, proves \text{\bf Case II}. Therefore the proof of Lemma \ref{le:orthogonality-in-strichartz} is complete.
\end{proof}

\begin{remark}\label{re-1}
With this lemma \ref{le:orthogonality-in-strichartz}, we have the following orthogonality result: for $(\alpha, q,r)$ defined as in Lemma \ref{le:orthogonality-in-strichartz} and $l\ge 1$,
$$\limsup_{n\to\infty}\|D^\alpha \sum_{j=1}^l e^{-(t-t_n^j)\partial_x^3}g_n^j \phi^j\|^q_{L^q_tL^r_x}\le
\sum_{j=1}^l\limsup_{n\to\infty}\|D^\alpha  e^{-(t-t_n^j)\partial_x^3}g_n^j \phi^j\|^q_{L^q_tL^r_x}$$
for $q\le r$; while for $r\le q$,
$$\limsup_{n\to\infty}\|D^\alpha \sum_{j=1}^l e^{-(t-t_n^j)\partial_x^3}g_n^j \phi^j\|^r_{L^q_tL^r_x}\le
\sum_{j=1}^l\limsup_{n\to\infty}\|D^\alpha  e^{-(t-t_n^j)\partial_x^3}g_n^j \phi^j\|^r_{L^q_tL^r_x}.$$
See \cite{Shao:2008:maximizers-Strichartz-Sobolev-Strichartz} for a similar proof.
\end{remark}

\section{Existence of extremisers}\label{sec:existence}
In this section we apply the linear profile decomposition Theorem \ref{thm:profile L8L8} to prove the existence of extremisers for \eqref{eq:airy-strichartz}.
\begin{proof}
Choose an extremising sequence $(f_n)_{n\ge 1}$ such that $$\|f_n\|_2=1, \, \lim_{n\to\infty} \|e^{-t\partial_x^3} f_n\|_{L^8_{t,x}}=\mathcal{A}.$$
By applying the linear profile decomposition in Theorem \ref{thm:profile L8L8}, we see that there is a sequence of profiles $\phi^j$ and errors $w_n^l$ such that for all $l\in\mathbb{N}$, up to a subsequence (in $n$),
\begin{equation*}
f_n=\sum_{1\le j\le l} e^{t_n^j\partial_x^3}g_n^j(\phi^j)+w_n^l .
\end{equation*}
Moreover,
\begin{align*}
  \mathcal{A}^8
 &=
  \lim_{n\to\infty} \|e^{-t\partial_x^3}f_n\|^8_{L^8_{t,x}}
  =
  \lim_{l\to\infty}\lim_{n\to\infty}
  \|\sum_{j=1}^l e^{-(t-t_n^j)\partial_x^3}g_n^j(\phi^j)\|^8_{L^8_{t,x}} \\
 &=
  \sum_{j=1}^\infty \|e^{-t\partial_x^3}\phi^j\|^{8}_{L^8_{t,x}}
  \le
  \mathcal{A}^8 \sum_{j=1}^\infty \|\phi^j\|^{2\times 4}_2
  \le
  \mathcal{A}^8 \left(\sum_{j=1}^\infty \|\phi^j\|^2_2 \right)^4
  \le \mathcal{A}^8.
\end{align*}
where the second equality follows from \eqref{eq:error}, the third equality from \eqref{eq:orthogonal2}, the first inequality from the definition of $\mathcal{A}$,
and the last inequality from $\sum_j\|\phi^j\|_2^2\le 1$, see Remark \ref{re-4}.

Thus the equal signs at the beginning and at the end force all the signs in this chain to be equal. Hence, we have
$$ 1=\bigl(\sum_{j=1}^\infty \|\phi^j\|^{2\times 4}_2\bigr)^{1/4}
\le \sum_{j=1}^\infty \|\phi^j\|^2_2 \le 1  $$
Thus
\begin{equation}\label{eq-equality}
 \bigl(\sum_{j=1}^\infty \|\phi^j\|^{2\times 4}_2\bigr)^{1/4} =
 \sum_{j=1}^\infty \|\phi^j\|^2_2
\end{equation}
which in turn implies that there is exactly one $j$ remaining. Without loss of generality, we may assume that
 $$\phi^j=0, \text{ for } j\ge 2. $$
Thus $\phi^1$ is an extremiser as desired.
\end{proof}
\begin{remark} The reason that \eqref{eq-equality} implies that at most one $\|\phi^j\|_2\neq 0$ is the strict concavity of $0\le s\mapsto s^\alpha$ for $0<\alpha<1$ (in particular, $\alpha=1/4$). More simply, if $0<\alpha<1$ then for $s_1,s_2\ge 0$ the inequality
\begin{equation}\label{eq:example}
(s_1+s_2)^\alpha \le s_1^\alpha + s_2^\alpha
\end{equation}
holds and if equality holds then either $s_1=0$ or $s_2=0$.
Indeed, one has
$$
  (s_1+s_2)^\alpha
 =
  \frac{s_1+s_2}{(s_1+s_2)^{1-\alpha}}
 =
  \frac{s_1}{(s_1+s_2)^{1-\alpha}} + \frac{s_2}{(s_1+s_2)^{1-\alpha}}
 \le
  s_1^\alpha +s_2^\alpha
$$
since $1-\alpha>0$ and the inequality is strict if both $s_1,s_2>0$.
\end{remark}
\begin{remark}\label{re-2}
Combining this argument with the orthogonality in Remark \ref{re-1}, the existence of extremisers for any non-endpoint Strichartz inequality can be obtained similarly. We omit the details here.
\end{remark}

\section{Analyticity of extremisers}\label{sec:analyticity}
In this section, we establish that any extremiser $f$ to \eqref{eq:airy-strichartz} enjoys an exponential decay in the Fourier space, Theorem \ref{thm-exp-decay}, from which the property of analyticity of extremisers follows easily. We begin with a bilinear Airy Strichartz estimate.
\begin{lemma}[Bilinear Airy estimates]\label{le:bilinear_airy}
 Suppose $\supp \widehat{f_1} \subset \{\xi: |\xi|\le N_1\}$ and $\supp \widehat{f_2} \subset \{\xi: N_2\le |\xi|\le 2N_2\}$, and $N_1\ll N_2$. Then
$$\|e^{-t\partial_x^3}f_1 e^{-t\partial_x^3} f_2\|_{L^4_{t,x}}\le C \bigl(\frac {N_1}{N_2}\bigr)^{1/4} \|f_1\|_2\|f_2\|_2.$$
where the constant $C>0$ is independent of $N_1$ and $N_2$.
\end{lemma}
\begin{proof}
We observe that
\begin{equation}\label{eq-15}
\|e^{-t\partial_x^3}f_1 e^{-t\partial_x^3} f_2\|_{L^4_{t,x}}=\|\int e^{ix(\xi_1+\xi_2)+it(\xi_1^3+\xi_2^3)}\widehat{f}_1(\xi_1)\widehat{f}_2(\xi_2)d\xi_1d\xi_2\|_{L^4_{t,x}}.
\end{equation}
We restrict the region to $\{(\xi_1, \xi_2): \xi_1,\xi_2\ge 0\}$ and change variables $a:=\xi_1+\xi_2$ and $b:=\xi_1^3+\xi_2^3$; then we see that the Jacobian $J\sim N_2^2$ since $N_1\ll N_2$. We apply the Hausdorff-Young inequality and changes of variables to see that \eqref{eq-15} is bounded by
\begin{align*}
&\lesssim \left(\iint J^{-1/3} |\widehat{f}_1\widehat{f}_2|^{4/3}d\xi_1d\xi_2\right)^{3/4}\\
& \lesssim |J|^{-1/4} \|f_1\|_2N_1^{1/4} \|f_2\|_2 N_2^{1/4} \\
&\lesssim \bigl(\frac {N_1}{N_2}\bigr)^{1/4} \|f_1\|_2\|f_2\|_2.
\end{align*}
\end{proof}

\begin{corollary}\label{cor-bilinear}
 If $\supp \widehat{f}_1 \subset \{|\xi_1|\le s\}$ and $\supp \widehat{f}_2 \subset \{|\xi_2|\ge Ls\}$ for some $s>1$ and $L\gg 1,$ then
\begin{equation}\label{eq-17}
\|e^{-t\partial_x^3}f_1 e^{-t\partial_x^3} f_2\|_{L^4_{t,x}}\le CL^{-1/4} \|f_1\|_2\|f_2\|_2.
\end{equation}
where the constant $C>0$ is independent of $L$.
\end{corollary}
\begin{proof}
Let $P_k$ denote the Littlewood-Paley projection operator to the frequency $\{2^k\le |\xi|\le 2^{k+1}\}$ for any $k\in \mathbb{Z}$. We dyadically decompose $f_2 =\sum_{k:\, 2^{k+1} \ge Ls }P_k f_2$. Then by the triangle inequality and Lemma \ref{le:bilinear_airy},
\begin{equation}\label{eq-55}
\begin{split}
 \|e^{-t\partial_x^3}f_1 e^{-t\partial_x^3} f_2\|_{L^4_{t,x}}
&\le
 \sum_{k:\,2^{k+1}\ge Ls} \|e^{-t\partial_x^3}f_1 e^{-t\partial_x^3}P_kf_2\|_{L^4_{t,x}}\\
&\lesssim
 \sum_{k:\, 2^{k+1} \ge Ls }
 \left(\frac {s}{2^k}\right)^{1/4} \|f_1\|_{L^2}\|P_k f_2\|_{L^2}\\
&\lesssim
 \|f_1\|_{L^2} s^{1/4} \sum_{k:\, 2^{k+1} \ge Ls } 2^{-k/4} \|P_k f_2\|_{L^2} \\
&\lesssim
 \|f_1\|_{L^2} s^{1/4}
 \left( \sum_{k:\,2^{k+1}\ge Ls} 2^{-k/2}\right)^{1/2} \left(\sum_{k}
 \|P_k f_2\|^2_{L^2}\right)^{1/2}\\
&\lesssim
 \|f_1\|_{L^2} s^{1/4} (Ls)^{-1/4} \|f_2\|_{L^2}\lesssim L^{-1/4} \|f_1\|_{L^2} \|f_2\|_{L^2}.
\end{split}
\end{equation}
This finishes the proof of Corollary \ref{cor-bilinear}.
\end{proof}

We define an $8$-linear form,
\begin{equation}\label{eq-18}
 Q(f_1,\cdots, f_8)
 :=
 \iint \Pi_{l=1}^4\overline{(e^{-t\partial_x^3}f_l)} \, \Pi_{m=5}^8(e^{-t\partial_x^3}f_m)dtdx.\end{equation}
where $f_i\in L^2,\,1\le i\le 8$. By the Airy Strichartz inequality \eqref{eq:airy-strichartz},
\begin{equation}\label{eq-9}
 |Q|\lesssim \Pi_{i=1}^8\|f_i\|^8_2.
 \end{equation}
Inspired by the Euler-Lagrange equation \eqref{eq:euler-lagrange}, we define the notion of weak solutions.
\begin{definition} A function $f\in L^2$ is said to be a weak solution to the Euler-Lagrange equation \eqref{eq:euler-lagrange} if it satisfies the following integral equation
\begin{equation}\label{eq:weak-solution-integral-form}
\omega \langle g, f\rangle =Q(g,f,\cdots, f), \quad \forall\, g\in L^2.
\end{equation} for some $\omega>0$. Here $\langle \cdot,\,\cdot\rangle $ is the inner product in $L^2$ defined by $\langle g,f\rangle =\int_\mathbb{R} \overline{g}f dx$.
\end{definition}

\begin{remark}
In view of the Euler-Lagrange equation \eqref{eq:euler-lagrange}, we see that, any extremiser $f$ to the Airy Strichartz inequality \eqref{eq:airy-strichartz} is actually a weak solution, as any solution $f$ of \eqref{eq:euler-lagrange} satisfies
\begin{equation}\label{eq-8}
\omega\langle g, f\rangle =Q(g,f,\cdots, f),\text{ with }\omega=\mathcal{A}^8\|f\|^6_2.
\end{equation}
\end{remark}
Now we list some additional notations and observations that are used in the following sections: Set
\begin{align}
 \label{eq-24} a(\eta)
&:=
 \sum_{l=1}^4\eta_l^3 -\sum_{m=5}^8\eta_m^3,\\
 \label{eq-25} b(\eta)
&:=
 \sum_{l=1}^4\eta_l- \sum_{m=5}^8\eta_m ,\\
\label{eq-26} M(h_1,\cdots, h_8)
&:=
\int_{\mathbb{R}^8} \Pi_{j=1}^8 |h_j(\eta_j)|\, \delta\bigl(a(\eta)\bigr)\delta\bigl(b(\eta)\bigr)d\eta,
\end{align}where $\delta$ denotes the Dirac mass. Then using the Fourier transform to represent $e^{-t\partial_x^3} f$ and doing the $t$ and $x$ integrals in the definition of $Q$, using $(2\pi)^{-1}\int e^{isr}dr = \delta(s)$ as distributions, we rewrite $Q$ as
\begin{equation}\label{eq-27}
Q(f_1,\cdots, f_8)=(2\pi)^{-3}\int_{\mathbb{R}^8}
 \Pi_{l=1}^4 \overline{\widehat{f_l}}(\eta_l) \, \Pi_{m=5}^8 \widehat{f_m}(\eta_m)
 \, \delta\bigl(a(\eta)\bigr)\delta\bigl(b(\eta)\bigr) d\eta.
\end{equation}
Then it is not hard to see that
\begin{align}
\label{eq-28}  Q(f_1,\cdots, f_8)& \le (2\pi)^{-3} M(|\widehat{f}_1|, \cdots, |\widehat{f}_8|),\\
\label{eq-29} M(h_1,\cdots, h_8) &= (2\pi)^3 Q(|h_1|^\vee, \cdots, |h_8|^\vee),
\end{align}
where
$ f^\vee (x):=(2\pi)^{-1/2} \int_\mathbb{R} e^{ix\xi} \widehat{f}(\xi)d\xi $
is the inverse Fourier transform.

Now we define a weighted version of $M$, for any function $F: \mathbb{R}\to \mathbb{R}$,
\begin{equation}\label{eq-33}
M_F(h_1,\cdots, h_8)
:=
\int_{\mathbb{R}^8} e^{F(\eta_1)-\sum_{l=2}^8F(\eta_l)} \Pi_{j=1}^8
|h_j(\eta_j)|\, \delta\bigl(a(\eta)\bigr)\delta\bigl(b(\eta)\bigr) d\eta.
\end{equation}
Then
\begin{equation}\label{eq-34}
M(e^F h_1, e^{-F}h_2,e^{-F}h_3,e^{-F}h_4,e^{-F}h_5,e^{-F}h_6,e^{-F}h_7,e^{-F}h_8)=M_F(h_1,\cdots, h_8).
\end{equation}
We define,  for $\mu>0, \,\eps>0$,
\begin{equation}\label{eq-14}
F_{\mu,\eps}(k):=\frac {\mu|k|^3}{1+\eps|k|^3} .
\end{equation}
\begin{proposition}\label{prop-sublinear}
For $F_{\mu,\eps}$ defined as above, we have
\begin{equation}\label{eq-35}
M_{F_{\mu,\eps}}(h_1,\cdots,h_8) \le M(h_1,\cdots, h_8)
\end{equation}
for all $\mu,\eps\ge 0$.
\end{proposition}
\begin{proof}
We see that the claim \eqref{eq-35} reduces to proving
$$ F_{\mu,\eps}(\eta_1)\le \sum_{l=2}^8 F_{\mu,\eps}(\eta_l), \text{ when } a(\eta)=b(\eta)=0$$
since then
$ e^{F_{\mu,\eps}(\eta_1)-\sum_{l=2}^8F_{\mu,\eps}(\eta_l)} \le e^0 =1$.
In fact, we only need $a(\eta)=0$ for this to hold.

Since $a(\eta)=0$ implies $\eta_1^3=\sum_{l=2}^8 (-1)^l \eta_l^3$,
\begin{equation}
\begin{split}
F_{\mu,\eps}(\eta_1) =\mu\frac {|\eta_1|^3}{1+\eps |\eta_1|^3}&=\mu\frac {|\sum_{l=2}^8 (-1)^l \eta_l^3|}{1+\eps|\sum_{l=2}^8 (-1)^l\eta_l^3|}\le \mu\frac {\sum_{l=2}^8 |\eta_l^3|}{1+\eps \sum_{l=2}^8  |\eta_l^3|}\\
&=\sum_{l=2}^8 \frac {\mu |\eta_l|^3}{1+\eps \sum_{l=2}^8  |\eta_l|^3} \le \sum_{l=2}^8 F_{\mu,\eps}(\eta_l),
\end{split}
\end{equation}where we have used the fact that $t\mapsto \frac {t}{1+\eps t} $ is increasing on $[0,\infty)$.
\end{proof}
\begin{remark}
 From the proof we can easily see that Proposition \ref{prop-sublinear} remains true if $F_{\mu,\eps}$ is replaced by $F$ where $F(k)= \widetilde{F}(|k|^3)$ with $\widetilde{F}$ increasing and $\widetilde{F}(a+b)\le \widetilde{F}(a)+ \widetilde{F}(b)$ for $a,b\ge 0$. Thus Proposition \ref{prop-sublinear} holds for a much larger class of functions than the one given in \eqref{eq-14}. However, for our goal of proving Theorem \ref{thm-exp-decay},  the class of functions in \eqref{eq-14} is the one we need.
\end{remark}

Combining \eqref{eq-9}, \eqref{eq-29}, and Corollary \ref{cor-bilinear} with Proposition \ref{prop-sublinear} and Parseval's identity, we can easily deduce
\begin{corollary} There exist a constant $C>0$ such that for $F_{\mu,\eps}$ defined as above and all $\mu,\eps\ge 0$
\begin{equation}\label{eq-37}
M_{F_{\mu,\eps}}(h_1,\cdots, h_8) \le C\, \Pi_{j=1}^8 \|h_j\|_2
\end{equation}
for all $h_j\in L^2$, j=1,\ldots,8. Moreover
\begin{equation}\label{eq-38}
M_{F_{\mu,\eps}}(h_1,\cdots, h_8)\le C L^{-1/4} \Pi_{j=1}^8 \|h_j\|_2
\end{equation}
provided that there exists at least one $h_j$ supported on $[-s,s]$ and another $h_k$ supported on $[-Ls, Ls]^c$ where $L\gg 1$ and $s\ge 1$.
\end{corollary}

\begin{remark}
The bounds \eqref{eq-37} and \eqref{eq-38} are surprising since a-priori it is not clear from the definition \eqref{eq-33} whether there exists an unbounded function $F$ such that $M_F$ is bounded on $L^2$. It is even more surprising that for the super-quadratic function $F(k)= \mu|k|^3$, the corresponding $M_F$ is bounded on $L^2$ for all $\mu\ge 0$ with a constant \emph{independent} of $\mu$. As the proof of Proposition \ref{prop-sublinear} shows this stems from the fact that in the definition of $M_F$ one integrates over the subset
$\{\eta\in\mathbb{R}^8:\, a(\eta)=0\}$ of $\mathbb{R}^8$. That restrictions in the integration to subspaces can lead to the boundedness of exponentially twisted functionals similar to $M_F$ was probably noticed first in  \cite{Erdogan-Hundertmark-Lee:2010:exponential-decay-dispersion-management-solitons}.
\end{remark}

The following proposition is the key to the proof of Theorem \ref{thm-exp-decay}. Let $F_{\mu,\eps}$ be defined as above for some $\eps>0, \mu>0$. Let $s>1$, we set
\begin{equation}
 \widehat{f}_>:= \widehat{f}1_{[-s^2,s^2]^c},\text{ and }\|\widehat{f}\|_{\mu,s,\eps}:= \|e^{F_{\mu,\eps}}\widehat{f}_>\|_2,
\end{equation}where $1_\Omega$ denotes the indicator function of the set $\Omega$.

\begin{proposition}\label{prop-bootstrap}
If $f$ is a weak solution to the Euler-Lagrange equation \eqref{eq:euler-lagrange} as defined in \eqref{eq:weak-solution-integral-form} with $\|f\|_2=1$. Then for $\mu=s^{-6}$ with $s\gg 1$, there exists a constant $C>0$ such that
\begin{equation}\label{eq-39}
\omega \|\widehat{f}\|_{s^{-6},\,s,\,\eps}\le o_1(1) \|\widehat{f}\|_{s^{-6},\,s,\,\eps}+C\sum_{l=2}^7 \|\widehat{f}\|^l_{s^{-6},\,s,\,\eps,}+o_2(1),
\end{equation} where $o_i(1)\to 0$ uniformly in $\eps >0$ as $s\to\infty$, $i=1,2$; the constant $C>0$ is independent of $\eps$ and $s$.
\end{proposition}
Let us postpone the proof of this proposition to the next section and finish the proof of Theorem \ref{thm-exp-decay}.

\begin{proof}[Proof of Theorem \ref{thm-exp-decay}] Given Proposition \ref{prop-bootstrap}, the proof is similar to the proof of exponential decay of dispersion management solitions given in \cite{Erdogan-Hundertmark-Lee:2010:exponential-decay-dispersion-management-solitons}.
We set
\begin{equation*}
G(v):=\frac \omega 2 v-C\sum_{l=2}^7 v^l, \text{ for } v\ge 0.
\end{equation*}
Invoking \eqref{eq-39}, if choosing $s$ large enough such that $o_1(1)\le \omega/2$, we obtain
\begin{equation}
G\bigl(\|\widehat{f}\|_{s^{-6},\,s,\,\eps}\bigr)\le o_2(1).
\end{equation} We observe that the graph of $G$ is concave in $[0,\infty)$ and intersects the $x$-axis only at two points: $v=0$ and $v=x_0$ for some $x_0>0$. Let $v_0,v_1>0$ such that $G(v_0)=G(v_1)=G_{\max}/2$, where $G_{\max}=\max\{G(v):~v\ge 0\}$. Again we take $s$ to be large enough such that $o_2(1)\le G_{\max}/2$. Then we have a dichotomy,
\begin{equation}\label{eq-40}
\text{ either }\|\widehat{f}\|_{s^{-6},s,1}\le v_0, \text{ or } \|\widehat{f}\|_{s^{-6},s,1}\ge v_1.
\end{equation} However the second choice is impossible if $s$ is chosen to be large, because by definition
$$F_{s^{-6},1}(k)=\frac {s^{-6}|k|^3}{1+|k|^3}\le s^{-6}\le 1,$$
which yields
$$\|\widehat{f}\|_{s^{-6},s,1} =\|e^{F_{s^{-6},1}}\widehat{f}_>\|_2\le e^{s^{-6}}\|\widehat{f}1_{[-s^2,s^2]^c}\|_2 \to 0, \text{ as } s\to \infty.$$
Now we fix such a large $s>0$ and consider the function $\eps \mapsto \|\widehat{f}\|_{s^{-6},s,\eps}$, which is continuous by the dominated convergence theorem for $\eps>0$. Again by \eqref{eq-39},
\begin{equation}\label{eq-42}
G\bigl(\|\widehat{f}\|_{s^{-6},s,\eps}\bigr)\le G_{\max}/{2}
\end{equation}
for all $\eps>0$. Hence by continuity, we must have that $\|\widehat{f}\|_{s^{-6},\,s,\,\eps}$ is in the same connected component of $G^{-1}\bigl([0,G_{\max}/2]\bigr)=[0,v_0]\cup [v_1,\infty)$. On the other hand, since we already know that $\|\widehat{f}\|_{s^{-6},s,1}\in [0,v_0]$, we deduce that
\begin{equation}\label{eq-43}
\|\widehat{f}\|_{s^{-6},\,s,\,\eps}\in [0,v_0], \,\forall\, \eps>0.
\end{equation}
This implies, by the monotone convergence theorem,
\begin{equation}
\|\widehat{f}\|_{s^{-6},\,s,\,0}=\lim_{\eps \to 0} \|\widehat{f}\|_{s^{-6},\,s,\,\eps}\le v_0.
\end{equation}
In other words,
\begin{equation}\label{eq-44}
 e^{s^{-6}|\cdot|^3}\widehat{f} 1_{[-s^2,s^2]^c}\in L^2
\end{equation}
and since $e^{s^{-6}|\cdot|^3}$ is bounded on $[-s^2,s^2]$ this yields
\begin{equation}\label{eq-45}
k\mapsto e^{s^{-6}|k|^3}\widehat{f}(k) \in L^2.
\end{equation}Let $\mu_0=s^{-6}$ for this $s>0$. Then the super Gaussian decay in Theorem \ref{thm-exp-decay} is established.

We are left with proving that $f$ is an entire function on the complex plane $\mathbb{C}$. Indeed, by the Cauchy-Schwarz inequality, for any $\mu\in \mathbb{R}$, we have
\begin{equation}\label{eq-23}
e^{\mu |k|}\widehat{f}(k)=e^{\mu|k|-\mu_0|k|^3} e^{\mu_0 |k|^3} \widehat{f}(k) \in L^1(\mathbb{R}),
\end{equation}
Then for any $z\in \mathbb{C}$, we can always choose $\mu>|z|$ such that
\begin{equation}
f(z)=(2\pi)^{-1/2} \int e^{izk}\widehat{f}(k)dk=(2\pi)^{-1/2} \int e^{izk-\mu|k|}e^{\mu|k|} \widehat{f}(k)dk.
\end{equation} Since  the first factor $e^{izk-\mu|k|}$ is bounded and the second factor is in $L^1$ by \eqref{eq-23}, $f$ is an entire function.
\end{proof}
It remains to prove Proposition \ref{prop-bootstrap}, which we carry out in the next section.

\section{The bootstrap argument}\label{sec:bootstrap}
In this section, we prove Proposition \ref{prop-bootstrap}, for which we only have the definition of weak solutions in \eqref{eq:weak-solution-integral-form} and the definition of $Q$ at our disposal.
We set $F=F_{\mu,\eps}$ for $F_{\mu,\eps}$ defined in \eqref{eq-14} and define
$f_>$, $h$, and $h_>$ by
\begin{equation}
\widehat{f}_>=\widehat{f}1_{[-s^2,s^2]^c},\,{h}(k)=e^{F(k)}\widehat{f}(k),\,
{h}_>(k):=e^{F(k)}\widehat{f}_>(k).
\end{equation}
\begin{proof}[Proof of Proposition \ref{prop-bootstrap}.]
We use $g=e^{2F(P)}f_>$ with $P=-i\partial_x$ in \eqref{eq:weak-solution-integral-form}. Using that the operator $e^{F(P)}$ is simply multiplying with $e^{2F(k)}$ in Fourier space, the representation \eqref{eq-27} of $Q$, and $h^\vee$ for the inverse Fourier transform of $h$, one sees
\begin{equation}\label{eq-16}
\begin{split}
\omega \|e^{F}\widehat{f}_>\|^2_2
&=\omega \langle e^{F}\widehat{f}_>(k),e^{F}\widehat{f}_>(k)\rangle
=\omega \langle e^{2F}\widehat{f}_>,\widehat{f}\rangle
=
\omega \langle e^{2F(P)}f_>,f\rangle \\
&=Q(e^{2F(P)}f_>, f,f,f,f,f,f,f) =Q((e^{F}h_>)^\vee, f,f,f,f,f,f,f)\\
&=Q((e^{F}h_>)^\vee,(e^{-F}h)^\vee,(e^{-F}h)^\vee,(e^{-F}h)^\vee,(e^{-F}h)^\vee,(e^{-F}h)^\vee,(e^{-F}h)^\vee,(e^{-F}h)^\vee)\\
&=:Q_F.
\end{split}
\end{equation}
Then by \eqref{eq-28}
\begin{equation}\label{eq-46}
|Q_F|\le C M_F(h_>,h,h,h,h,h,h,h)\le CM(h_>,h,h,h,h,h,h,h),
\end{equation} where the last inequality follows from Proposition \ref{prop-sublinear}. Continuing \eqref{eq-46}, we split $h$ and use that the operator $M$ is sublinear in each component,
\begin{equation}\label{eq-49}
\begin{split}
M(h_>,h,h,h,h,h,h,h)&\le M(h_>,h_<,h_<,h_<,h_<,h_<,h_<,h_<)+\\
&\quad+\sum_{j_2,\cdots,j_8\in \{>,<\}, \atop \text{at least one } j_l=~>} M(h_>,h_{j_2},h_{j_3},h_{j_4},h_{j_5},h_{j_6},h_{j_7},h_{j_8})=:A+B.
\end{split}
\end{equation}
We split further $h_<=h_{\ll}+h_\sim$, where
the low frequency part $\widehat{h}_\ll:=\widehat{h}1_{[-s,s]}$ and the median frequency part $\widehat{h}_\sim:= \widehat{h}1_{[-s^2,s^2]\setminus [-s,s]}$.

We estimate $A$ by using the bilinear Airy Strichartz estimate in Lemma \ref{le:bilinear_airy}:
\begin{equation}\label{eq-50}
\begin{split}
A=&M(h_>,h_<,h_<,h_<,h_<,h_<,h_<,h_<)\\
&\le M(h_>,h_\ll,h_<,h_<,h_<,h_<,h_<,h_<)+M(h_>,h_\sim,h_<,h_<,h_<,h_<,h_<,h_<) \\
&\le s^{-1/4}\|h_>\|_2 \|h_\ll\|_2\|h_<\|_2^6+ \|h_>\|_2 \|h_\sim\|_2\|h_<\|_2^6\\
&=\|h_>\|_2\bigl( s^{-1/4}\|h_\ll\|_2+\|h_\sim\|_2\bigr)\|h_<\|_2^6.
\end{split}
\end{equation}
Recalling that $\|f\|_2=1$, then
\begin{equation}
\begin{split}
\|h_<\|_2&=\|e^{F_{\mu,\eps}} \widehat{f}_<\|_2\le \|e^{\mu|k|^3} \widehat{f}_<\|_2\le e^{\mu s^6}\|f\|_2,\\
\|h_\ll\|_2 &=\|e^{F_{\mu,\eps}}\widehat{f}_\ll\|_2\le e^{\mu s^3}\|f\|_2,\\
\|h_\sim\|_2 &=\|e^{F_{\mu,\eps}}\widehat{f}_\sim\|_2 \le e^{\mu s^6}\|f_\sim\|_2,
\end{split}
\end{equation} we obtain
\begin{equation}\label{eq-51}
A\le C \|h_>\|_2 \bigl(s^{-1/4}e^{\mu s^3-\mu s^6} +\|f_\sim\|_2\bigr)e^{7\mu s^6}.
\end{equation}

Now we turn to estimate $B$.
\begin{equation}
\begin{split}
B&\le \sum_{j_2,\cdots,j_8\in \{>,<\}, \atop \text{exactly one } j_l=~>} M(h_>,h_{j_2},h_{j_3},h_{j_4},h_{j_5},h_{j_6},h_{j_7},h_{j_8})+\\
&\quad+\sum_{j_2,\cdots,j_8\in \{>,<\}, \atop \text{two and more } j_l=~>} M(h_>,h_{j_2},h_{j_3},h_{j_4},h_{j_5},h_{j_6},h_{j_7},h_{j_8})=:B_1+B_2.
\end{split}
\end{equation}
For $B_2$,
\begin{equation}
B_2\lesssim \|h_>\|_2\Pi_{l=2}^8 \|h_{j_l}\|_2
\lesssim \|h_>\|_2 \left( \sum_{l=2}^7 \|h_<\|^{7-l}_2\|h_>\|_2^l\right)
\lesssim \|h_>\|_2 \,e^{5\mu s^6}\sum_{l=2}^7 \|h_>\|_2^l
\end{equation} where we have used that $\|h_<\|_2\lesssim e^{\mu s^6} \|f_<\|_2\lesssim e^{\mu s^6}$.

For $B_1$, we split one of the $h_<$ into $h_<=h_\ll+h_\sim $ and then use the sublinearity of $M$,
\begin{equation}
\begin{split}
B_1\lesssim \|h_>\|_2 &\left(s^{-1/4}\|h_\ll\|_2+\|h_\sim\|_2\right)\|h_<\|^5_2\|h_>\|_2\\
&\lesssim \|h_>\|^2_2 \left(s^{-1/4}e^{\mu s^3-\mu s^6}+\|f_\sim\|_2\right)e^{\mu s^6}\|h_<\|^5_2\\
&\quad\quad\lesssim \|h_>\|^2_2 \left(s^{-1/4}e^{\mu s^3-\mu s^6}+\|f_\sim\|_2\right)e^{6\mu s^6}.
\end{split}
\end{equation}
Thus we conclude that
\begin{equation}\label{eq-52}
B\le B_1+B_2\lesssim \|h_>\|^2_2 \left(s^{-1/4}e^{\mu s^3-\mu s^6}+\|f_\sim\|_2\right)e^{6\mu s^6}+ e^{5\mu s^6}\|h_>\|_2\sum_{l=2}^7 \|h_>\|_2^l.
\end{equation}
Therefore from \eqref{eq-16}, \eqref{eq-46}, \eqref{eq-49}, \eqref{eq-51} and \eqref{eq-52}, we have
\begin{equation}\label{eq-53}
\begin{split}
\omega \|\widehat{h}_>\|_2^2&\lesssim  \|h_>\|_2 \bigl(s^{-1/4}e^{\mu s^3-\mu s^6} +\|f_\sim\|_2\bigr)e^{7\mu s^6}+\\
 &+\|h_>\|^2_2 \left(s^{-1/4}e^{\mu s^3-\mu s^6}+\|f_\sim\|_2\right)e^{6\mu s^6}+ e^{5\mu s^6}\|h_>\|_2\sum_{l=2}^7 \|h_>\|_2^l
 \end{split}
\end{equation}Canceling one $\|\widehat{h}_>\|_2$ on both sides, we see that
\begin{equation}\label{eq-54}
\begin{split}
\omega \|\widehat{h}_>\|_2&\lesssim  \bigl(s^{-1/4}e^{\mu s^3-\mu s^6} +\|f_\sim\|_2\bigr)e^{7\mu s^6}+\\
 &\|h_>\|_2\left(s^{-1/4}e^{\mu s^3-\mu s^6}+\|f_\sim\|_2\right)e^{6\mu s^6}+ e^{5\mu s^6}\sum_{l=2}^7 \|h_>\|_2^l
 \end{split}
\end{equation}Since $ \|f_\sim \|_2=\|\widehat{f}1_{[-s^2,s^2]\setminus [-s,s]}\|_2 \le \|\widehat{f}1_{[-s,s]^c}\|_2 =o(1)$ as $s\to\infty$, and $e^{6\mu s^6}=e^6$ if taking $\mu=s^{-6}$, we conclude that
\begin{equation}
\omega \|h_>\|_2\le o_1(1)\|h_>\|_2+C\sum_{l=2}^7 \|h_>\|_2^l+o_2(1).
\end{equation} Therefore the proof of Proposition \ref{prop-bootstrap} is complete.
\end{proof}

\textbf{Acknowledgements.} The research was carried out when S. Shao visited the math department at the University of Illinois, Urbana Champaign, and he was deeply grateful for its hospitality. D. Hundertmark was supported by NSF grant DMS-0803120. During the early preparation of this work, S. Shao was supported by the National Science Foundation under agreement DMS-0635607. Any opinions, findings, and conclusions or recommendations expressed in this paper are those of the authors and do not necessarily reflect the views of the National Science Foundation.


\end{document}